\documentclass[pdflatex,sn-mathphys-num]{sn-jnl}% Math and Physical Sciences Numbered Reference Style 
\usepackage{graphicx}%
\usepackage{multirow}%
\usepackage{amsmath,amssymb,amsfonts}%
\usepackage{amsthm}%
\usepackage{mathrsfs}%
\usepackage[title]{appendix}%
\usepackage{xcolor}%
\usepackage{textcomp}%
\usepackage{manyfoot}%
\usepackage{booktabs}%
\usepackage{algorithm}%
\usepackage{algorithmicx}%
\usepackage{algpseudocode}%
\usepackage{listings}%

\usepackage{siunitx}
\usepackage{anyfontsize}

\newcommand{\ve}{\varepsilon}

\newcommand{\mL}{\mathcal{L}}

\newcommand{\mN}{\mathcal{N}}

\newcommand{\ps}{\partial_s}
\newcommand{\pt}{\partial_t}

\newcommand{\Dp}{\mathrm{D}_p}
\newcommand{\nx}{\nabla_x}

\newcommand{\RR}{\mathbb{R}}
\newcommand{\EE}{\mathbb{E}}

\newcommand{\tp}{^{\top}}
\newcommand{\abs}[1]{\left|#1\right|}
\newcommand{\parentheses}[1]{\left(#1\right)}
\newcommand{\sqbra}[1]{\left[#1\right]}

\newcommand{\pd}[2]{\dfrac{\partial #1}{\partial #2}}

\newcommand{\rd}{\mathrm{d}}
\newtheorem{theorem}{Theorem}%  meant for continuous numbers
\newtheorem{proposition}[theorem]{Proposition}% 
\newtheorem{assumption}{Assumption}
\newtheorem{remark}{Remark}%

\begin{document}

\title[Forward-backward score dynamics]{A deep learning algorithm for computing mean field control problems via forward-backward score dynamics}


% \author[Zhou]{Mo Zhou}
% \email{mozhou366@math.ucla.edu}
% \address{Department of Mathematics, University of California, Los Angeles}
% \author[Osher]{Stanley Osher}
% \email{sjo@math.ucla.edu}
% \address{Department of Mathematics, University of California, Los Angeles}
% \author[Li]{Wuchen Li}
% \email{wuchen@mailbox.sc.edu}
% \address{Department of Mathematics, University of South Carolina, Columbia.}

% \author*[1,2]{\fnm{First} \sur{Author}}\email{iauthor@gmail.com}



% \affil*[1]{\orgdiv{Department}, \orgname{Organization}, \orgaddress{\street{Street}, \city{City}, \postcode{100190}, \state{State}, \country{Country}}}

% \affil[2]{\orgdiv{Department}, \orgname{Organization}, \orgaddress{\street{Street}, \city{City}, \postcode{10587}, \state{State}, \country{Country}}}

\author[1]{\fnm{Mo} \sur{Zhou}}\email{mozhou366@math.ucla.edu}

\author[1]{\fnm{Stanley} \sur{Osher}}\email{sjo@math.ucla.edu}

\author*[2]{\fnm{Wuchen} \sur{Li}}\email{wuchen@mailbox.sc.edu}

\affil[1]{\orgdiv{Department of Mathematics}, \orgname{University of California}, \orgaddress{\city{Los Angeles}, \postcode{90095}, \state{CA}, \country{USA}}}

\affil[2]{\orgdiv{Department of Mathematics}, \orgname{University of South Carolina}, \orgaddress{\city{Columbia}, \postcode{29208}, \state{SC}, \country{USA}}}

% \affil[3]{\orgdiv{Department}, \orgname{Organization}, \orgaddress{\street{Street}, \city{City}, \postcode{610101}, \state{State}, \country{Country}}}


\abstract{We propose a deep learning approach to compute mean field control problems with individual noises. The problem consists of the Fokker-Planck (FP) equation and the Hamilton-Jacobi-Bellman (HJB) equation. Using the differential of the entropy, namely the score function, we first formulate the deterministic forward-backward characteristics for the mean field control system, which is different from the classical forward-backward stochastic differential equations (FBSDEs). We further apply the neural network approximation to fit the proposed deterministic characteristic lines. Numerical examples, including the control problem with entropy potential energy, the linear quadratic regulator, and the systemic risks, demonstrate the effectiveness of the proposed method. }




\maketitle

\section{Introduction}

Mean field control (MFC) and mean field games (MFG) have emerged as powerful frameworks in the realm of control theory, offering a versatile approach to model and optimize large-scale systems with a multitude of interacting components among agents. This paradigm finds applications across various domains, such as economics and finance \cite{carmonaa2023deep}, herd behavior \cite{bauso2016opinion}, robotics \cite{elamvazhuthi2019mean}, and data science \cite{lin2020apac}. Its significance lies in its ability to address complex systems with a large number of interacting agents by considering their mean field behaviors, thereby reducing the computational burden associated with analyzing each individual agent.

Along with the MFC problem, two important equations characterize the optimality condition of the system. The first one is the Fokker-Planck (FP) equation, which describes the evolution of the density of the system. The second one is the Hamilton--Jacobi-Bellman (HJB) equation, which describes the optimal expected cost w.r.t. certain initialization. Solving the coupled FP--HJB system below is crucial for solving the MFC problem
\begin{equation*}
\left\{\begin{aligned}
&\pt\rho(t,x) + \nabla_x\cdot \parentheses{\rho(t,x) \Dp H(t,x,\nabla_x\phi(t,x))}=\frac1\beta \Delta_x \rho(t,x),\\
&\pt\phi(t,x)+H(t,x,\nabla_x\phi(t,x))+\frac1\beta \Delta_x\phi(t,x) = f(t,x,\rho(t,x)).
\end{aligned}\right.
\end{equation*}
Here the initial condition is $\rho(0,x)=\rho_0(x)$ and the terminal condition is $\phi(T,x)=-V(x)$, where $\rho_0\colon \mathbb{R}^d\rightarrow\mathbb{R}$ is a probability density function and $V\colon \mathbb{R}^d\rightarrow\mathbb{R}$ is a function. $H(t,x,p)\colon\mathbb{R}\times\mathbb{R}^d\times\mathbb{R}^d\rightarrow \mathbb{R}$ is the Hamiltonian of the system, $x\in\mathbb{R}^d$ and $p\in\mathbb{R}^d$ are the space and adjoint variables respectively. $\Dp$ denotes the gradient of $H$ w.r.t. adjoint variable $p$. And $1/\beta>0$ is a diffusion coefficient parameter, which characterizes the intensity of diffusion. The function $f\colon\mathbb{R}\times\mathbb{R}^d\times\mathbb{R}\rightarrow\mathbb{R}$ involves the density $\rho$, which represents the interaction behaviors of agents. The density dependency in Hamilton-Jacobi equations distinguishes the MFC problem from the standard stochastic control problem. 
The details for the system are introduced in Section \ref{sec:background}.

One traditional way to solve the MFC problem is through the FBSDE system, which utilizes the stochastic characterization of the problem. The viscous state dynamic, also known as the stochastic characteristic curve, is often studied together with the backward adjoint equation that describes the cost \cite{carmona2018probabilistic} or shadow price \cite{yong1999stochastic}.
This gives the classic FBSDE system \cite{zhou1997hamiltonian}. The theoretical analysis for this method is elegant. However, these approaches heavily rely on the sampling of Brownian motions (BM), which may give large numerical errors from random trajectories.

In recent years, the score function has emerged as a powerful tool in many problems, such as the generative model and the MFC. The score function is the gradient of the logarithm of the density function $\nabla\log\rho(t,x)$. Based on this score function, the probability flow \cite{boffi2023probability} is defined to characterize the stochastic state dynamic with a deterministic ODE, potentially alleviating the issue of sampling the Brownian motion.

In this paper, we propose a new framework to solve the MFC problem based on the score function. This is to formulate deterministic characteristics lines for the Fokker-Planck equations and Hamilton-Jacobi equations. In detail, we denote $x_t$ as the probability flow of the state, which satisfies an ODE whose velocity is subtracted by the score function.
We also define $y_t = \phi(t,x_t)$ as an analog for the adjoint process.  
%%that is commonly used in FBSDE.
These two dynamics form a forward-backward score ODE system, which is summarized below. 
\begin{equation}\label{eq:ODE1}
\left\{ \begin{aligned}
\pt x_t &= \Dp H(t,x_t,z_t) - \frac1\beta \nx \log \rho(t,x_t),\\
\pt y_t &= L(t,X_t,\Dp H(t,x_t,z_t)) + f_t - \frac1\beta h_t - \frac1\beta z_t\tp \nx \log \rho(t,x_t),
\end{aligned} \right.
\end{equation}
with given initial condition $x_0$ and terminal condition $y_T = -V(x_T)$, where $V(\cdot)$ is the terminal cost. Here, $z_t$, $h_t$, and $f_t$ are short for $\nx\phi(t,x_t)$, $\Delta_x\phi(t,x_t)$, and $f(t,x_t,\rho(t,x_t))$ respectively. And $\phi$ will be parametrized as a neural network later. A detailed description of system \eqref{eq:ODE1} is given in Section \ref{sec:background}. Clearly, when $\beta=\infty$, equation \eqref{eq:ODE1} is the classical deterministic characteristic for the HJB without viscosity. 

In the algorithm, we parameterize $\phi$ as a neural network and construct a loss function to match the adjoint process $y_t$ with its neural network parametrization. Meanwhile, we also formulate a density estimation method to approximate the gradient of the logarithm of the density function. This density estimation distinguishes our method from those that require sampling of Brownian motions.
% This step requires density estimation, which is different from sampling Brown motions. 
%The most important feature of our ODE system is that it gets rid of sampling Brownian motions. 
The validity of the proposed method is justified through several numerical examples including MFC problems with an entropy potential energy, linear quadratic problems, and system risks.

Many studies have been conducted on MFG/MFC problems. Some of the works focus on the McKean--Vlasov equation for the system \cite{carmona2015forward,de2021backward,germain2022numerical,han2022learning}. Others focus on controlling the physical and social systems, arising from conservation laws \cite{li2022controlling,li2023controlling} and reaction-diffusion equations \cite{li2022computational}. In recent years, there has been a growing trend to use machine learning approaches to compute these problems, such as \cite{angiuli2022unified, beck2019machine,ruthotto2020machine,zhou2021actor}. There is also research studying convergence properties of these machine learning methods \cite{carmona2021convergence,carmona2022convergence}. There are also other algorithms such as the Picard iterations method \cite{ji2020three}.
Many of the works mentioned above rely heavily on the theory of FBSDE \cite{ma1999forward}. This FBSDE is usually computed through the shooting method, where one matches the terminal condition with its approximation. However, such methods rely heavily on the sampling of Brownian motion, which only has an accuracy of order one-half. In numerical examples, our deterministic score dynamic provides an alternative for this problem. We leave the theoretical comparison for algorithms using forward-backward score dynamics and FB SDEs in future works. 

The rest of this paper is organized as follows. 
Section \ref{sec:background} provides a theoretical background for the MFC problem and introduces our forward-backward ODE score system. In Section \ref{sec:numerical_alg}, we present a detailed description of our numerical method, including the construction of the loss function, accompanied by theoretical justification, the estimation for the density function, and the numerical discretization. Finally, Section \ref{sec:example} shows numerical examples that validate the effectiveness of our proposed algorithm, including the MFC problem with an entropy potential energy, the linear quadratic regulator, and the systemic risks.

\section{Formulation of the score-based mean field control problem}\label{sec:background}

In this section, we briefly review the MFC problem and formulate the deterministic forward-backward ODE system, with the score functions to represent the viscosity for the density function.

We consider the MFC problem described by the cost functional:
\begin{equation}\label{eq:cost_density}
\inf_v \int_0^T \int_{\RR^d} \parentheses{L(t,x,v(t,x)) \rho(t,x) + F(t,x,\rho(t,x))} \,\rd x \,\rd t + \int_{\RR^d} V(x) \rho(T,x) \,\rd x\,
\end{equation}
where $L\colon \RR \times \RR^d \times \RR^d \to \RR$ is the Lagrangian of the system, $v\colon \RR \times \RR^d \to \RR^d$ is the velocity field, $F\colon \RR \times \RR^d \times \RR \to \RR$ is the running cost that involves the density, and $V \colon \RR^d \to \RR$ is the terminal energy.
% \wc{The definition of $L$, terminal energy $V$.}
Here, the density $\rho(t,x)$ satisfies the Fokker-Planck equation
\begin{equation}\label{eq:FokkerPlanck}
\pt\rho(t,x) + \nabla_x\cdot(\rho(t,x)v(t,x))=\frac1\beta \Delta_x \rho(t,x),\quad \rho(0,x)=\rho_0(x).
\end{equation}
The objective is to minimize the cost functional \eqref{eq:cost_density} over all feasible velocity fields of control.

We give the following assumption for the system.
\begin{assumption}\label{assump:convex_L}
The Lagrangian $L(t,x,v)$ is strongly convex in $v$.
\end{assumption}
This assumption guarantees the well-posedness of the MFC problem, which is necessary in our setting. In more general control problems, where the drift is nonlinear in the control field and the volatility is not constant, well-posedness may be guaranteed even if the Lagrangian is not convex. The well-posedness is determined by the concavity (or convexity, depending on the definition,) of the generalized Hamiltonian \cite{yong1999stochastic,zhou2023policy}. 

Under Assumption \ref{assump:convex_L}, the Hamiltonian, defined as the Legendre transform of the Lagrangian, attains a well-defined expression
\begin{equation}\label{eq:Hamiltonian}
H(t,x,p) = \sup_{v \in \RR^d} v\tp p - L(t,x,v).
\end{equation}
Applying standard results in convex analysis, we know that the maximum in \eqref{eq:Hamiltonian} is realized at \begin{equation}\label{eq:optimality}
v^* = \Dp H(t,x,p),
\end{equation}
where $\Dp$ refers to the gradient w.r.t. $p$ (cf. \cite{villani2021topics} chapter 2).

Next, we analyze the MFC problem by introducing a Lagrange multiplier, $\phi(t,x)$, resulting in the following expression
\begin{equation}\label{eq:Lagrange_multiplier}
\begin{aligned}
\inf_{v} \sup_\phi & \int_0^T \int_{\RR^d} 
\left[ L(t,x,v(t,x)) \rho(t,x) + F(t,x,\rho(t,x)) \right.\\
& \hspace{0.5in} + \phi(t,x)\left. \left(\pt\rho(t,x) + \nabla_x\cdot(\rho(t,x)v(t,x)) - \frac1\beta \Delta_x \rho(t,x) \right)\right] \,\rd x \,\rd t \\
& + \int_{\RR^d} V(x) \rho(T,x) \,\rd x
\end{aligned}
\end{equation}
Solving the critical point system for the above inf-sup problem yields the following proposition.

\begin{proposition}\label{prop:FP_HJB}
Let Assumption \ref{assump:convex_L} hold. The optimal velocity field for the MFC problem is given by 
\begin{equation*}
v(t,x) = \Dp H(t,x,\nabla_x\phi(t,x)).
\end{equation*}
Here $\phi\colon [0,T]\times \mathbb{R}^d\rightarrow\mathbb{R}$ is a function that satisfies the following FP--HJB system with the density function $\rho(t,x)$

% Under Assumption \ref{assump:convex_L}, there exists a function $\phi\colon [0,T]\times \mathbb{R}^d\rightarrow\mathbb{R}$, such that the optimal velocity field satisfies 
% \begin{equation*}
% v(t,x) = \Dp H(t,x,\nabla_x\phi(t,x)),
% \end{equation*}
% and the solution of the MFC is described by the FP--HJB system
\begin{equation}\label{eq:FP_HJB}
\left\{\begin{aligned}
&\pt\rho(t,x) + \nabla_x\cdot \parentheses{\rho(t,x) \Dp H(t,x,\nabla_x\phi(t,x))}=\frac1\beta \Delta_x \rho(t,x),\\
&\pt\phi(t,x)+H(t,x,\nabla_x\phi(t,x))+\frac1\beta \Delta_x\phi(t,x) = f(t,x,\rho(t,x)),\\
& \rho(0,x)=\rho_0(x),\quad \phi(T,x)=-V(x), 
\end{aligned}\right.
\end{equation}
where $f(t,x,\rho(t,x)) = \pd{F}{\rho}(t,x,\rho(t,x))$.
\end{proposition}
\begin{proof}
For illustration purposes, we take the functional derivative of the inf-sup problem \eqref{eq:Lagrange_multiplier} for all the variables and derive the critical point system for the problem.
Taking the derivative of \eqref{eq:Lagrange_multiplier} w.r.t. $\phi$ directly results in the Fokker Planck equation \eqref{eq:FokkerPlanck}. To obtain other equations for the critical point system, we perform integration by part and reformulate \eqref{eq:Lagrange_multiplier} as 
\begin{equation}\label{eq:Lagrange_multiplier2}
\begin{aligned}
\inf_{v} \sup_\phi &\int_0^T \int_{\RR^d}
\left[ L(t,x,v(t,x)) \rho(t,x) + F(t,x,\rho(t,x)) - \pt\phi(t,x) \rho(t,x)  \right.\\
& \qquad \qquad \left.  - \nabla_x \phi(t,x)\tp v(t,x) \rho(t,x) - \frac1\beta \Delta_x \phi(t,x) \rho(t,x) \right] \,\rd x \,\rd t \\
&+\int_{\RR^d} \parentheses{V(x) \rho(T,x) + \phi(T,x) \rho(T,x) - \phi(0,x) \rho(0,x)} \,\rd x.
\end{aligned}
\end{equation}
Taking derivative of \eqref{eq:Lagrange_multiplier2} w.r.t. $\rho$, we obtain
$$L(t,x,v(t,x)) - \pt\phi(t,x) - \nabla_x \phi(t,x)\tp v(t,x) - \frac1\beta \Delta_x \phi(t,x) + \pd{F}{\rho}(t,x,\rho(t,x)) = 0,$$
which simplifies to
\begin{equation}\label{eq:HJB_proof}
\pt\phi(t,x) + \nabla_x \phi(t,x)\tp v(t,x) - L(t,x,v(t,x)) + \frac1\beta \Delta_x \phi(t,x) = f(t,x,\rho(t,x)).
\end{equation}
Taking derivative of \eqref{eq:Lagrange_multiplier2} w.r.t. $v$, we arrive at
$$\parentheses{\nabla_v L(t,x,v(t,x)) - \nx \phi(t,x)}\rho(t,x)=0,$$
which implies
\begin{equation}\label{eq:v_phix}
v(t,x) = \Dp H(t,x,\nx \phi(t,x))
\end{equation}
according to \eqref{eq:optimality}. Finally, taking derivative of \eqref{eq:Lagrange_multiplier2} w.r.t. $\rho(T,\cdot)$, we arrive at the terminal condition $\phi(T,x) = -V(x)$. If we plug \eqref{eq:v_phix} into the Fokker--Planck equation and \eqref{eq:HJB_proof}, we recover the FP--HJB system \eqref{eq:FP_HJB}. 
\end{proof}
\begin{remark}
The FP--HJB system in Proposition 1 corresponds to the characterization of MFC problem in \cite{bensoussan2013mean} section 4. Here, $\phi(t,x)$ solves the HJB equation, which characterizes optimality of the system. In contrast, the formulation for MFG problems does not include a central planner, and is therefore different. We refer the reader to \cite{bensoussan2013mean} section 3 for details. In an optimal control or MFG problem, the solution to the HJB equation not only captures optimality, but also refers to the value function, representing the expected cost (or reward) given a specific initial time and state. However, in MFC problems, the solution to the HJB equation is only a characterization of optimality and is not necessarily a value function. A value function for MFC problems can still be defined in terms of the distribution, referred to as the master equation. We refer the reader to \cite{cardaliaguet2019master} for details.
\end{remark}

As an illustrative example, consider the case where $L(t,x,v) = \frac12 \abs{v}^2$. This choice leads to $H(t,x,p) = \frac12 \abs{p}^2$, resulting in the following simplified FP--HJB system
\begin{equation}\label{eq:FP_HJB_LQ}
\left\{\begin{aligned}
&\pt\rho(t,x) + \nabla_x\cdot \parentheses{\rho(t,x) \nabla_x\phi(t,x)}=\frac1\beta \Delta_x \rho(t,x),\\
&\pt\phi(t,x) + \frac12 \abs{\nabla_x \phi(t,x)}^2 + \frac1\beta \Delta_x\phi(t,x) = f(t,x,\rho(t,x)),\\
& \rho(0,x)=\rho_0(x),\quad \phi(T,x)=-V(x). 
\end{aligned}\right.
\end{equation}
When $f=0$, the problem reduces to an optimal control one and the HJB equation becomes the viscous Burgers equation.

Traditional methods such as backward stochastic differential equations (BSDEs) often involve sampling for Brownian motion, leading to potential errors. In this work, we innovatively propose an alternative ODE system based on the score function. The probability flow \cite{boffi2023probability} associated with the Fokker-Planck equation \eqref{eq:FokkerPlanck} is defined as
\begin{equation}\label{eq:ODE_x_v}
\pt x_t = v(t,x_t) - \frac1\beta \nx \log \rho(t,x_t),
\end{equation}
where the density $\rho$ is involved. The term $\nabla_x \log \rho(t,x_t)$ is known as the score function, commonly used in generative models \cite{song2020score}. This probability flow constitutes an ODE system, with the only source of randomness being the initial condition $x_0$. If $x_0$ follows the distribution $\rho_0$, then the density for $x_t$ in \eqref{eq:ODE_x_v} precisely corresponds to the solution of the FP equation \eqref{eq:FokkerPlanck}.

According to Proposition \ref{prop:FP_HJB}, the probability flow under optimal velocity becomes the following ODE system
\begin{equation}\label{eq:probability_flow}
\pt x_t = \Dp H(t,x_t,\nx\phi(t,x_t)) - \frac1\beta \nx \log \rho(t,x_t).
\end{equation}

This deterministic system liberates us from the need for Brownian motion sampling. We next introduce our forward-backward ODE system.

\begin{proposition}
Let $\rho$, $\phi$ be the solution to the FP--HJB system \eqref{eq:FP_HJB}. Let $x_t$ be the probability flow defined by \eqref{eq:probability_flow}. Let $y_t = \phi(t,x_t)$. Then, $x_t$ and $y_t$ satisfies the forward-backward ODE system
\begin{equation}\label{eq:ODEsystem}
\left\{ \begin{aligned}
\pt x_t &= \Dp H(t,x_t,z_t) - \frac1\beta \nx \log \rho(t,x_t),\\
\pt y_t &= L(t,X_t,\Dp H(t,x_t,z_t)) + f_t - \frac1\beta h_t - \frac1\beta z_t\tp \nx \log \rho(t,x_t),
\end{aligned} \right.
\end{equation}
with initial condition $x_0 \sim \rho_0$ and terminal condition $y_T = -V(x_T)$.
Here, $f_t$ is short for $f(t,x_t,\rho(t,x_t))$, $z_t = \nx \phi(t,x_t)$, and $h_t = \Delta_x \phi(t,x_t)$.
\end{proposition}
\begin{proof}
The dynamic for $x_t$ comes from equation \eqref{eq:probability_flow}. 
For $y_t=\phi(t,x_t)$, we have
\begin{equation*}
\begin{aligned}
& \quad \pt y_t = \dfrac{\rd}{\rd t} \phi(t,x_t) = \pt \phi(t,x_t) + \nx \phi(t,x_t)\tp \pt x_t\\
& = f(t,x_t,\rho(t,x_t)) -\frac1\beta \Delta_x \phi(t,x_t) -  H\parentheses{t,x_t,\nx \phi(t,x_t)} \\
& \quad + \nx \phi(t,x_t)\tp \parentheses{\Dp H(t,x_t,\nx\phi(t,x_t)) - \frac1\beta \nx \log \rho(t,x_t)}\\
& = L(t,X_t,\Dp H(t,x_t,z_t)) + f_t - \frac1\beta h_t - \frac1\beta z_t\tp \nx \log \rho(t,x_t),
\end{aligned}
\end{equation*}
where we used the HJB equation and the probability flow in the second equality. The third equality is by the definition of Hamiltonian \eqref{eq:Hamiltonian}. This recovers the second ODE in \eqref{eq:ODEsystem}.
\end{proof}
\begin{remark}
This forward-backward ODE system can be viewed as a deterministic analog of the FBSDE, which is commonly studied in MFC problems. To present the FBSDE in our scenario, we first give the stochastic version of our MFC
\begin{equation*}%\label{eq:cost_exp}
\inf_{v} \EE \sqbra{ \int_0^T \parentheses{L(t,X_t,v(t,X_t)) + F(t,X_t,\rho(t,X_t)) / \rho(t,X_t)} \,\rd t + V(X_T)}
\end{equation*}
subject to 
\begin{equation}\label{eq:SDE_x}
\rd X_t = v(t,X_t) \,\rd t + \sqrt{2/\beta} \,\rd W_t \quad\quad X_0 \sim \rho_0.
\end{equation}
Here, $\rho(t,x)$ is the density function for $X_t$, $W_t$ is a standard Wiener process in $\RR^d$, and the expectation is taken over the whole trajectory $X_t \sim \rho(t,\cdot)$.
The optimal stochastic state dynamic \eqref{eq:SDE_x}, together with the adjoint dynamic for value,
% $$\rd Y_t = \parentheses{L(t,X_t,v(t,X_t)) + f(t,X_t,\rho(t,X_t))} \,\rd t + Z_t \,\rd W_t \quad\quad Y_T = V(X_T),$$
forms the FBSDE for our MFC problem
% In our scenario, the FBSDE is written as
\begin{equation}\label{eq:FBSDE}
\left\{\begin{aligned}
\rd X_t &= \Dp H(t,X_t,Z_t) \,\rd t + \sqrt{2/\beta}\rd W_t,   & X_0 &\sim \rho_0,\\
\rd Y_t &= \parentheses{L(t,X_t,\Dp H(t,X_t,Z_t)) + f(t,X_t,\rho(t,X_t))} \,\rd t + \sqrt{2/\beta} Z_t\tp \,\rd W_t, \quad  & Y_T &= -V(X_T).
\end{aligned}\right.
\end{equation}
%\wc{Brief explain what is $Z_t$. How does this trajectory similar and different from our score dynamics? In particular, does $Z_t$ depend on $\beta$?} 
Here, $Z_t$ is the auxiliary variable for the BSDE. The unique solution to the BSDE above is $Y_t = \phi(t,X_t)$ and $Z_t = \nx\phi(t,X_t)$. In this case, $(X_t, Y_t, Z_t)$ are based on the solutions of stochastic dynamics \eqref{eq:FBSDE}. In comparison, $(x_t, y_t, z_t)$ are solutions from equation system \eqref{eq:ODEsystem}. They interact with each other through the density of $x_t$. 
\end{remark}

\section{Numerical method}\label{sec:numerical_alg}
In this section, we apply a neural network parametrization to $\phi$ and construct a loss function to match the forward-backward ODE system. Then, we introduce the density estimation function based on the kernel method. 
% \wc{In this section, we apply the neural network functions to approximate the proposed deterministic forward-backforward dynamics. Loss function. Density estimation. Optimization steps. }

\subsection{Construction of loss function}
In practice, we can parametrize $\phi$ as a neural network $\phi_{\mN}$ and represent $z_t$ and $h_t$ through auto-differentiation. 
Then we construct a loss function that try to match the dynamic $y_t$ from  \eqref{eq:ODEsystem} and $\phi_\mN(t,x_t)$
\begin{equation}\label{eq:loss}
\mL = \EE \int_0^T \abs{\phi_\mN(t,x_t) - y_t}^2 \,\rd t,
\end{equation}
where the expectation is taken over $x_0 \sim \rho_0$. $x_t$ and $y_t$ are computed through the ODE system \eqref{eq:ODEsystem}, with $y_0 = \phi_\mN(0,x_0)$, $z_t = \nx \phi_\mN(t,x_t)$, and $h_t= 
\Delta_x \phi_\mN(t,x_t)$. 
\begin{proposition}
Let $\phi_\mN$ be a function such that $\phi_\mN(T,\cdot) = -V(\cdot)$. Let $\rho$ be the solution to the Fokker-Planck equation
\begin{equation*}
    \pt\rho(t,x) + \nabla_x\cdot \parentheses{\rho(t,x) \,\Dp H(t,x, \nabla_x\phi_\mN(t,x))}=\frac1\beta \Delta_x \rho(t,x),
\end{equation*}
with initial condition $\rho(0,\cdot) = \rho_0(\cdot)$. If the loss $\mL$ defined by \eqref{eq:loss} is $0$, then $\rho$ and $\phi_\mN$ is the solution to the FP--HJB system \eqref{eq:FP_HJB}.
\end{proposition}
\begin{proof}
According to the ODE system \eqref{eq:ODEsystem}, we obtain
\begin{align*}
 \phi_\mN(t,x_t) = &\phi_\mN(0,x_0) + \int_0^t \parentheses{ \ps \phi_\mN(s,x_s) + \nx \phi_\mN(s,x_s)\tp \ps x_s } \rd s\\
=& y_0 + \int_0^t \sqbra{ \ps \phi_\mN(s,x_s) + z_s\tp \parentheses{\Dp H(s,x_s,z_s) - \frac1\beta \nx \log \rho(s,x_s)} } \rd s.
\end{align*}
Also, $y_t$ satisfies
\begin{align*}
 y_t =& y_0 + \int_0^t \ps y_s \,\rd s \\
 =& y_0 + \int_0^t \left[ f_s - \frac1\beta \Delta_x \phi_\mN(s,x_s) - H(s,x_s,\nx \phi_\mN(s,x_s)) \right.\\
& \quad \quad \left. + z_s\tp\parentheses{\Dp H(s,x_s,\nx\phi_\mN(s,x_s)) - \frac1\beta \nx \log \rho(s,x_s)} \right] \rd s.
\end{align*}
Subtracting the above two equations, we obtain
$$\phi_\mN(t,x_t) - y_t = \int_0^t \parentheses{ \ps \phi_\mN(s,x_s) +  \frac1\beta \Delta_x \phi_\mN(s,x_s) + H(s,x_s,\nx \phi_\mN(s,x_s)) - f_s} \rd s.$$
So
$$\mL = \EE\int_0^T \sqbra{\int_0^t \parentheses{ \ps \phi_\mN(s,x_s) +  \frac1\beta \Delta_x \phi_\mN(s,x_s) + H(s,x_s,\nx \phi_\mN(s,x_s)) - f_s} \rd s}^2 \rd t.$$
Note that the integrand above is exactly the residual of the HJB equation, so $\mL=0$ implies that $\phi_\mN$ is the solution to the HJB equation. Therefore, $\rho$ and $\phi_\mN$ is the solution to the FP--HJB system \eqref{eq:FP_HJB}.
\end{proof}

\subsection{Density estimation}
Computing the density information is usually necessary and unavoidable in the MFG/MFC problems.
In this work, we consider the classic kernel density estimate (KDE) with a Gaussian Kernel.

Let 
$$\rho_{ref} = \dfrac1N \sum_{i=1}^N \delta_{x_i}$$
be the reference measure from sampling. We define the kernel density
\begin{equation*}
\widehat\rho(t,x) = K(\cdot,\sigma_K) \ast \rho_{ref} (x) =\int K(x-y,\sigma_K) \, \rho_{ref}(y) \, \rd y, 
\end{equation*}
where $K(z,\sigma_K)=(2\pi\sigma_K^2)^{-d/2} \exp(-|z|^2 / (2\sigma_K^2))$ is the Gaussian kernel with $\sigma_K>0$. The bandwidth $\sigma_K$ is a hyperparameter.

We remark that we have also tested other kernels such as the Wasserstein proximal kernel proposed in \cite{tan2023noise}. Its performance is similar to the Gaussian kernel. Thus, we just report Gaussian KDE in this work. 

\subsection{Numerical implementation of the score-based MFC solver}

In this section, we present the numerical implementation of the algorithm. 

The function $\phi(t,x)$ is parametrized by a neural network $\phi_\mN(t,x)$. As for the terminal condition $\phi(t,x)=-V(x)$, we give a hard parametrization through
\begin{equation}\label{eq:NN_terminal}
\phi_\mN(t,x) = \dfrac{T-t}{T} \mN(t,x) - \dfrac{t}{T} V(x),
\end{equation}
where $\mN(t,x)$ is a standard fully connected deep neural network. If an explicit terminal condition is absent (see the example in Section \ref{sec:SystemicRisk} below), we set $\phi_\mN(t,x) = \mN(t,x)$ and include an additional loss term to enforce the terminal condition.

We discretize the ODE system \eqref{eq:ODEsystem} using forward Euler scheme with $N_T+1$ nodes $\{t_j=j \Delta t \}_{j=0}^{N_T}$ and step size $\Delta t = T/N_T$,
\begin{equation}\label{eq:Euler_scheme}
\left\{ \begin{aligned}
x_{t_{j+1}} &= x_{t_j} + \parentheses{\Dp H(t_j,x_{t_j},z_{t_j}) - \frac1\beta \nx \log \widehat\rho(t_j,x_{t_j})} \Delta t\\
 y_{t_{j+1}} &= y_{t_j} + \parentheses{ f_{t_j} - \frac1\beta h_{t_j} - H(t_j,x_{t_j},z_{t_j}) + z_{t_j}\tp \Dp H(t_j,x_{t_j},z_{t_j}) - \frac1\beta z_{t_j}\tp \nx \log \widehat\rho(t_j,x_{t_j})} \Delta t
\end{aligned} \right.
\end{equation}
with $z_{t_j} = \nx \phi_\mN(t,x_{t_j})$ and $h_{t_j}=\Delta_x \phi_\mN(t,x_{t_j})$ obtained from the auto-differentiation, and $\widehat\rho(t,x)$ obtained from the KDE. Given $N$ sampled trajectories $\{x_t^{(i)}\}_{i=1}^N$, this KDE is given by
\begin{equation}\label{eq:KDE_discrete}
\widehat\rho(t,x) = \dfrac{1}{N} \sum_{i=1}^N K(x-x_t^{(i)},\sigma_K). 
\end{equation}
The loss function \eqref{eq:loss} is then discretized into
$$\dfrac{1}{N} \sum_{i=1}^N \sum_{j=1}^{N_T} \sqbra{\phi_\mN(t_j,x_{t_j}^{(i)}) - y_{t_j}^{(i)}}^2 \Delta t,$$
and we minimize this loss using the Adam optimization scheme. The algorithm is summarized in the pseudocode Algorithm \ref{alg}. The detailed parameters are deferred to the appendix.

\begin{algorithm}
\caption{Forward-backward score solver for the MFC problem}\label{alg}
\begin{algorithmic}
\Require MFC problem, $\sigma_K$, number of training steps $k_{end}$, parameters for Adam optimizer, batch size $N$, number of time intervals $N_T$
\Ensure the solution to the HJB equation
\State initialization: parameter for the network $\phi_\mN$
\For{$k=1$ \textbf{to} $k_{end}$}
\State Sample $N$ points $\{x_0^{(i)}\}_{i=1}^N$ from the initial distribution $\rho_0$
\State Compute $y_0^{(i)} = \phi_\mN(0,x_0^{(i)})$
\State Loss $\mL = 0$
\For{$j=0$ \textbf{to} $N_T-1$}
% \State $t = j \Delta t$
\State compute $\widehat\rho(t,x)$ through the KDE method \eqref{eq:KDE_discrete}
\State compute $\{z_{t_j}^{(i)}\}_{i=1}^N$ and $\{h_{t_j}^{(i)}\}_{i=1}^N$ through auto-differentiation
\State compute $\{x_{t_{j+1}}^{(i)}\}_{i=1}^N$ and $\{y_{t_{j+1}}^{(i)}\}_{i=1}^N$ through the Euler scheme \eqref{eq:Euler_scheme}
\State update loss $\mL = \mL + \dfrac{1}{N} \sum_{i=1}^N \parentheses{\phi_\mN(t_{j+1}, x^{(i)}_{t_{j+1}}) - y^{(i)}_{t_{j+1}}}^2 \Delta t$
\EndFor
\State update the parameters for $\phi_\mN$ through Adam method to minimize the loss $\mL$
\EndFor
\end{algorithmic}
\end{algorithm}

This algorithm can be generalized to the broader MFC setting, where the dependence of the HJB equation on the density $\rho$ is more intricate, as illustrated in Section \ref{sec:SystemicRisk}.

We will also compare our algorithm with the traditional BSDE method. In the BSDE approach, one sample the trajectories of the stochastic processes \eqref{eq:FBSDE} using Euler–Maruyama scheme instead of sampling \eqref{eq:ODEsystem} with scheme \eqref{eq:Euler_scheme}. The other part has the same implementation as our score method.
We also remark that in the BSDE method, a shooting loss that only try to match the terminal condition could be constructed. The performance for such shooting method is similar to matching the whole trajectory.

% Also, the loss function is constructed to match the terminal condition, which is the reason why it is called a shooting method. We also remark that a loss function to match the whole trajectory of $Y_t$ could also be constructed for the BSDE method, and its performance is similar to the shooting method.

\section{Numerical examples}\label{sec:example}
In this section, we substantiate the efficacy of our algorithm through the presentation of three numerical examples. The first case under consideration has previously been examined in \cite{lin2020apac}, which has a potential energy. The second example pertains to a linear quadratic mean field control problem, while the third example delves into the realm of systemic risk in finance. Specifically, it addresses the intricate dynamics associated with banks' interactions through borrowing and lending, as elucidated in \cite{carmona2018probabilistic}. The details for the numerical implementation are deferred to the appendix. We also compare our score method with the traditional BSDE method. %Our score method performs better than the BSDE method in two of the three examples.

\subsection{An MFC example with entropy potential energy}
We first consider an example that is also studied in \cite{lin2020apac}. The objective is  
\begin{equation*}
\inf_v \int_0^T \int_{\RR^d} \parentheses{\frac12\abs{v(t,x)}^2 + \frac12 |x|^2 + \gamma\log(\rho(t,x)) - \gamma} \rho(t,x) \,\rd x \,\rd t + \int_{\RR^d} V(x) \rho(T,x) \,\rd x
\end{equation*}
subject to the dynamic \eqref{eq:ODE_x_v} with initial distribution $\rho_0 \sim N(0,\frac{1}{\alpha} I_d)$. Here $\gamma$ and $\alpha$ are two positive constants.
The FP--HJB equation pair is
\begin{equation}
\left\{\begin{aligned}
&\pt\rho(t,x) + \nx\cdot\big(\rho(t,x)\nabla_x\phi(t,x)\big)= \Delta_x \rho(t,x),\\
&\pt\phi(t,x) + \Delta_x\phi(t,x) + \frac12\abs{\nabla_x\phi(t,x)}^2 - \frac12\abs{x}^2 - \gamma \log\rho(t,x) = 0,\\
& \rho(0,x)=\rho_0(x),\quad \phi(T,x)=-V(x). 
\end{aligned}\right.
\end{equation}
With proper terminal cost and choice of $\gamma$ and $\alpha$ that satisfy $\alpha^2 + \gamma \alpha = 1$ (i.e. $\alpha = \dfrac{-\gamma + \sqrt{\gamma^2 + 4}}{2}$), an analytical solution to the problem is obtained, given by
$$\phi(t,x) = \parentheses{d \alpha + \dfrac{\gamma d}{2}\log(\dfrac{\alpha}{2\pi})}t - \dfrac{\alpha \abs{x}^2}{2}$$
and
$$\rho(t,x) = \parentheses{\dfrac{\alpha}{2\pi}}^{d/2} \exp\parentheses{-\dfrac{\alpha \abs{x}^2}{2}}.$$
This solution is utilized to assess the error and validate our algorithm. The outcomes of our numerical experiments are depicted in Figure \ref{fig:static}. The results are presented separately for $1$ and $2$ dimensions.

In the first row, representing the $1$ dimensional case, the initial figure in the first column displays the training curve for the neural network. It illustrates the $L^2$ relative error of $\phi$, as well as its gradient and Laplacian, presented in logarithmic scale. Notably, approximating $\phi$ solely with $L^2$ norm is insufficient due to the essential information provided by $\nabla_x\phi$ for determining the optimal policy (cf. \eqref{eq:optimality} with $p=\nabla_x\phi(t,x)$). Hence, we also report the errors for the derivatives. The shadow in the figure represents the standard deviation observed during multiple test runs. The final errors for these tests are $1.59\%$, $0.52\%$, $0.51\%$ respectively.  % Score L2 error $6.13\%$
As a comparison, the final errors using the BSDE method are $2.77\%$, $2.63\%$, $2.36\%$, which is higher than our score method.
The second figure in the first row illustrates a comparison between the initial value $\phi(0,\cdot)$ and its neural network approximation. Recall that we enforce a rigorous implementation of the terminal condition $\phi(T,\cdot)$ within the network. The network adeptly captures the value function.
The third and fourth figures present comparison plots for the terminal density $\rho(T,\cdot)$ and score function $\nabla_x \log(\rho(T,x))$ with the corresponding KDE approximation from the trajectory \eqref{eq:probability_flow}. We also report that the Wasserstein-2 distance between our estimated samples and the true density is \num{6.41e-2} for our score method and \num{8.23e-2} for the BSDE method. 

Moving to the second row, which corresponds to the $2$ dimensional scenario, the initial figure in the first column represents the training curve. The final errors are $1.82\%$, $0.54\%$, $0.53\%$. As a comparison, the final errors using the BSDE method are $1.92\%$, $2.32\%$, $2.14\%$, which is higher than our score method.
Given the challenge of directly visualizing high-dimensional input functions, a density plot of $\phi(0,\cdot)$ is shown in the second figure. Specifically, this plot illustrates the density function of $\phi(0,x_0)$ and its neural network approximation, where $x_0$ is uniformly distributed in the box $[-2/\sqrt{\alpha},2/\sqrt{\alpha}] \times [-2/\sqrt{\alpha},2/\sqrt{\alpha}]$. Note that $2/\sqrt{\alpha}$ is twice the standard deviation of the initial distribution. Similar visualization techniques for high-dimensional functions have been employed in previous studies \cite{han2020solving}.
The third and fourth figures in this row compare the contour plots of the density $\rho(T,\cdot)$ and its KDE approximation. The final Wasserstein-2 error for the samples is \num{9.18e-2} for our score method and \num{5.87e-2} for the BSDE method. 

\begin{figure*}[t!]
    \centering
    \includegraphics[width=0.242\textwidth]{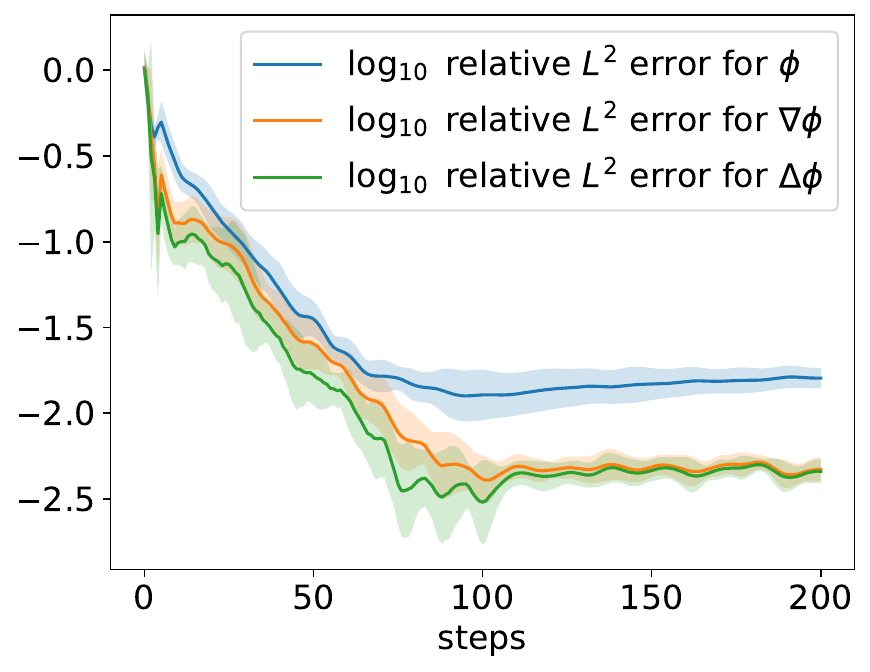}
    \includegraphics[width=0.242\textwidth]{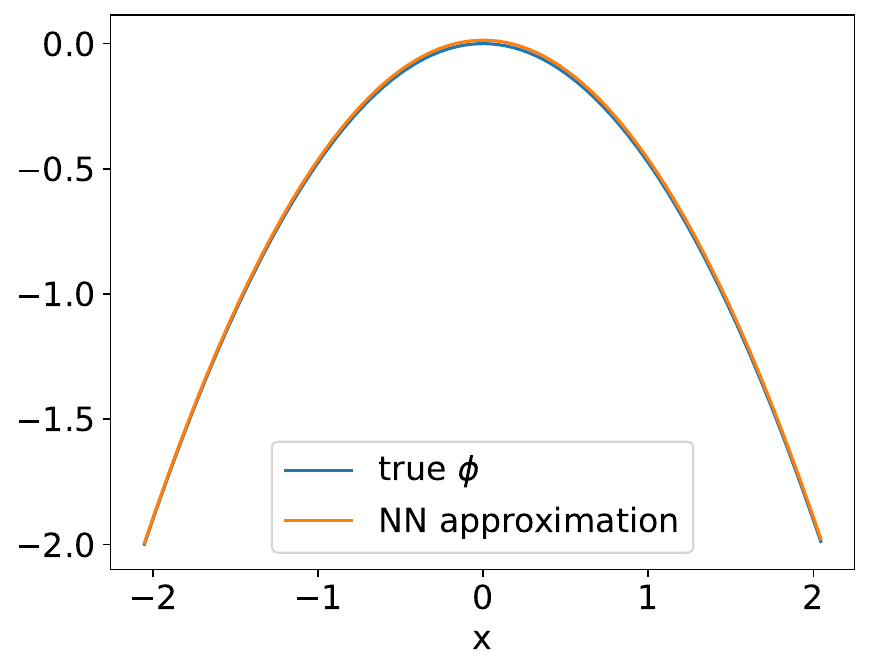}
    \includegraphics[width=0.242\textwidth]{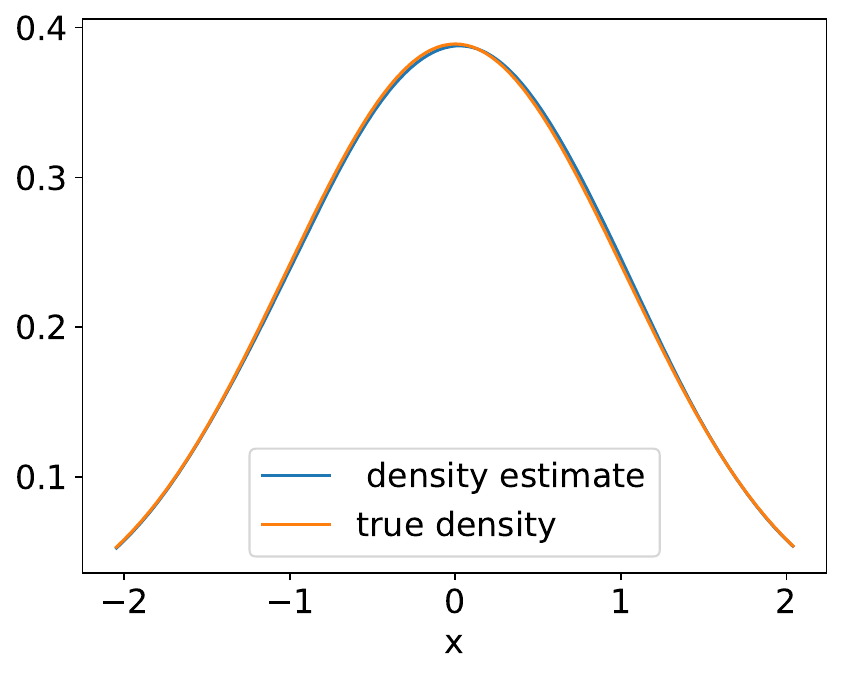}
    \includegraphics[width=0.242\textwidth]{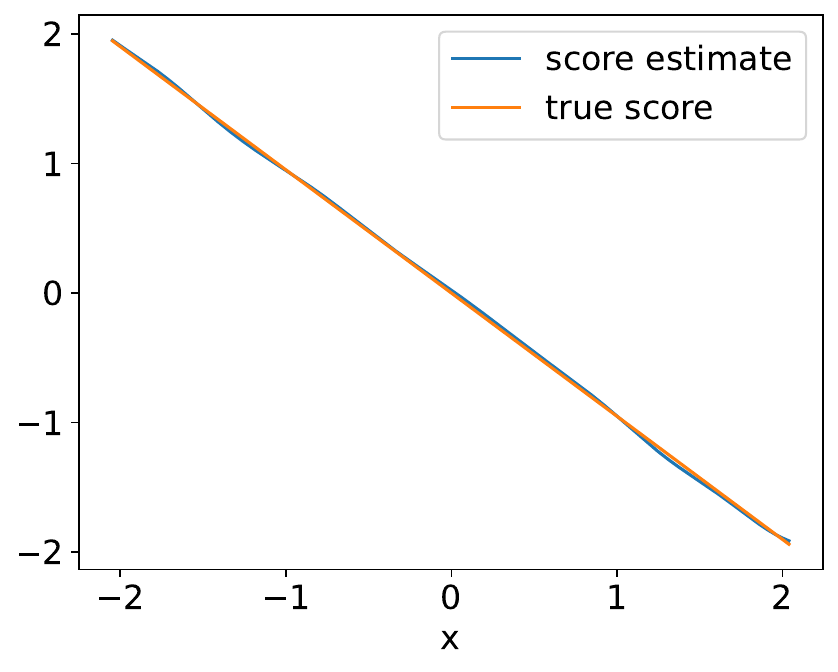}
    \includegraphics[width=0.242\textwidth]{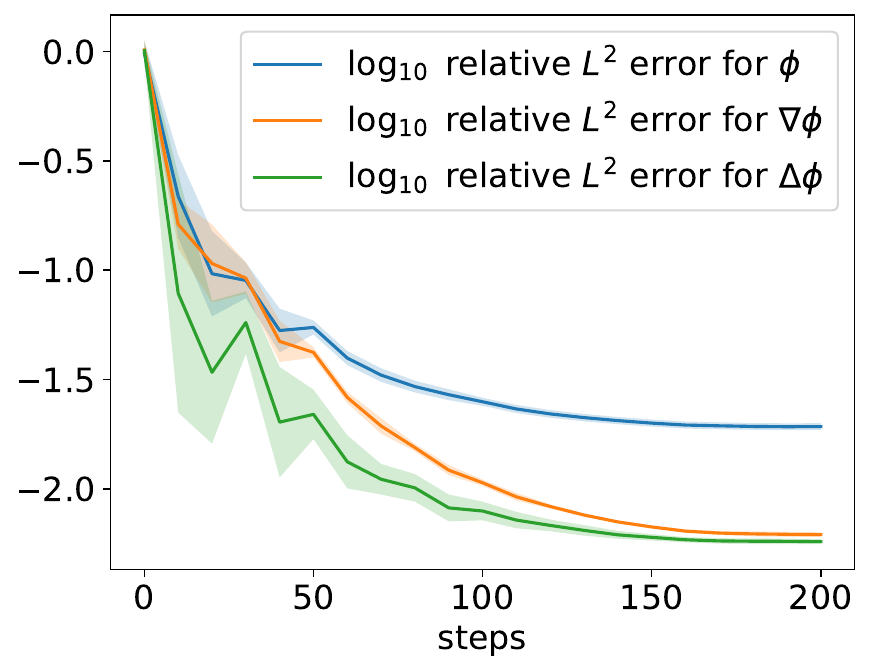}
    \includegraphics[width=0.242\textwidth]{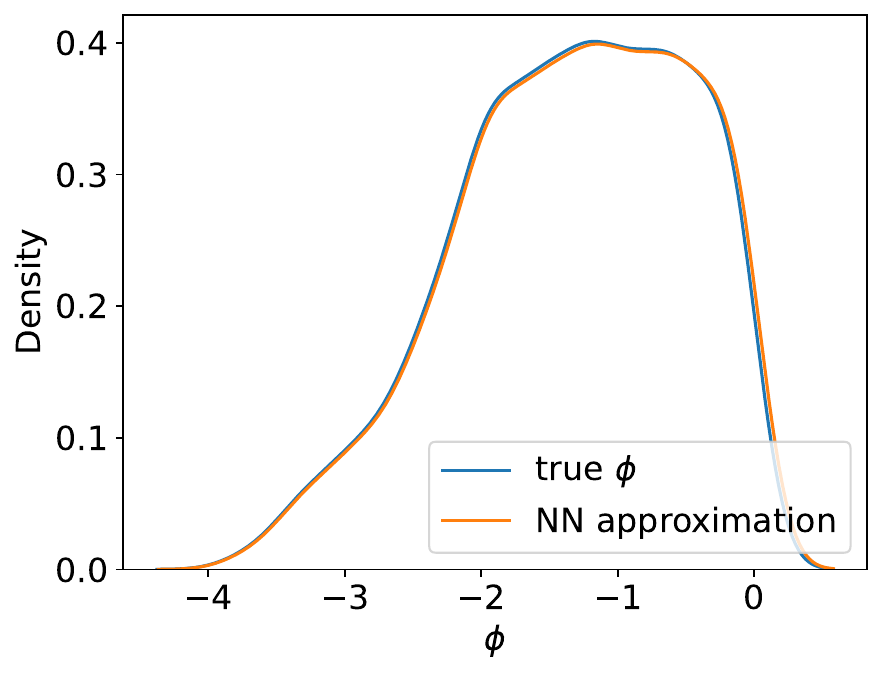}
    \includegraphics[width=0.49\textwidth]{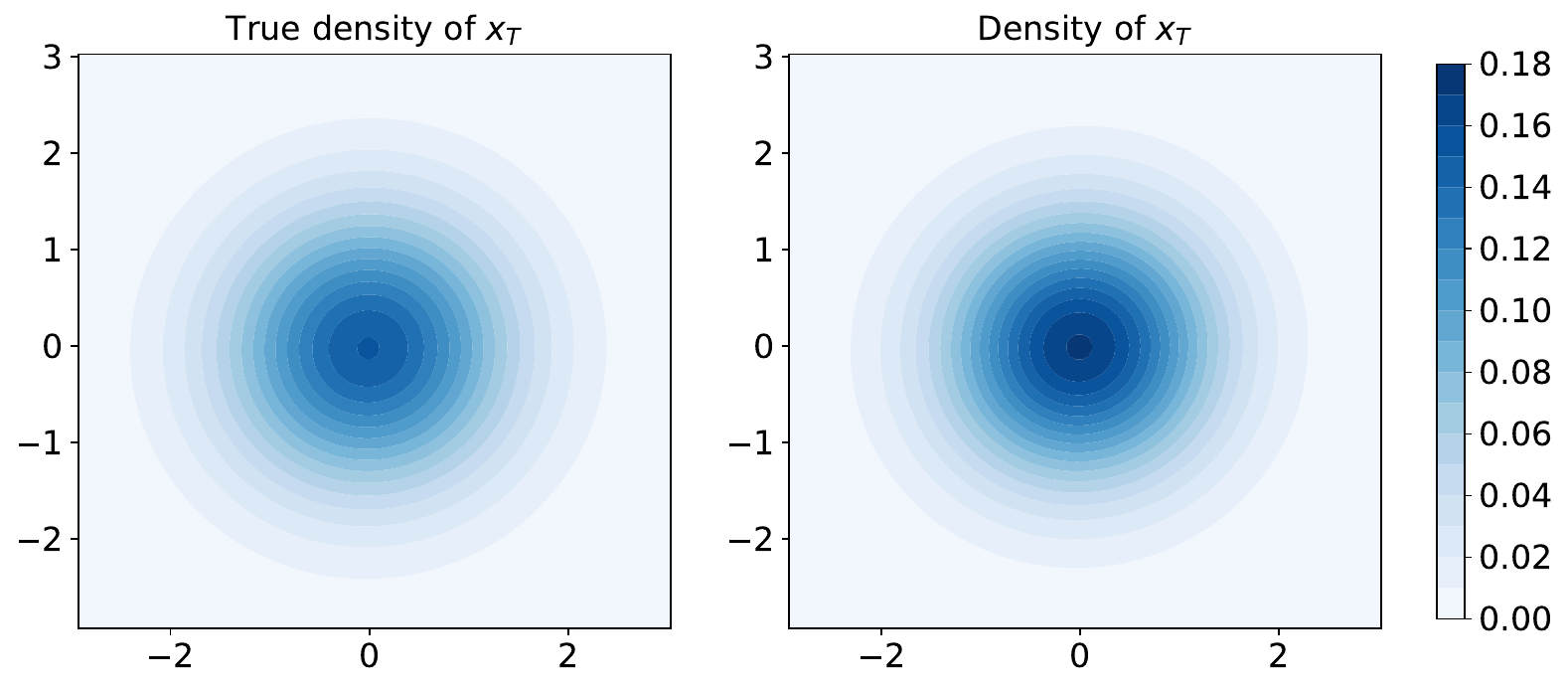}
    \caption{Numerical results for the MFC problem with log density running cost. The first row illustrates the results in $1$ dimension, including the training curve, plot of $\phi(0,\cdot)$, density plot, and score plot. The second row presents the results in $2$ dimensions, including the training curve, density plot for $\phi(0,x_0)$, and a contour plot comparison for the density $\rho(T,x_T)$.}
    \label{fig:static}
\end{figure*}

\subsection{An LQ example}
In the subsequent example, we explore a scenario involving linear dynamics and quadratic cost, commonly referred to as the linear quadratic (LQ) problem. The objective is defined by 
\begin{equation*}
\inf_v \int_0^T \int_{\RR^d} \parentheses{\frac12\abs{v(t,x)}^2 + \gamma\parentheses{\frac12 \rho(t,x) - \rho^*(t,x)}} \rho(t,x) \,\rd x \,\rd t + \int_{\RR^d} V(x) \rho(T,x) \,\rd x
\end{equation*}
subject to the dynamic \eqref{eq:ODE_x_v} with initial distribution $\rho_0 \sim N(0, 2(T+1) I_d / \beta)$.  In this example, the term with coefficient $\gamma$  is intentionally introduced to increase the difficulty of the problem. The FP--HJB system governing this problem is
\begin{equation*}
\left\{\begin{aligned}
&\pt\rho(t,x) + \nabla_x\cdot \parentheses{\rho(t,x) \nabla_x\phi(t,x)}=\frac1\beta \Delta_x \rho(t,x),\\
&\pt\phi(t,x) + \frac12 \abs{\nabla_x \phi(t,x)}^2 + \frac1\beta \Delta_x\phi(t,x) = \gamma\parentheses{\rho(t,x) - \rho^*(t,x)},\\
& \rho(0,x)=\rho_0(x),\quad \phi(T,x)=-V(x)=-|x|^2/2. 
\end{aligned}\right.
\end{equation*}
Here, $\rho^*$ is the true solution with the optimal trajectory, given by
\begin{equation*}
\rho^*(t,x) = \parentheses{4\pi (T-t+1)/\beta}^{-\frac{d}{2}} \exp\parentheses{-\dfrac{\beta\abs{x}^2}{4 (T-t+1)}}.
\end{equation*}
The solution to the HJB equation is 
\begin{equation*}
\phi(t,x) = \frac1\beta d\log\frac{1}{T-t+1}-\frac{\abs{x}^2}{2(T-t+1)}.
\end{equation*}
The numerical outcomes are presented in Figure \ref{fig:LQ}, with the first and second rows showing the results for $1$ and $2$ dimensions respectively.

In the $1$ dimensional scenario depicted in the first row, the first, second, and fourth figures mirror the format of Figure \ref{fig:static}, illustrating the learning curves, $\phi(0,\cdot)$ plot, and the score $\nabla_x \log(\rho(T,\cdot))$ plot. The final errors for $\phi$ and their derivatives are $2.00\%$, $2.76\%$, $2.75\%$.
The errors for the BSDE method are $2.14\%$, $2.84\%$, $2.83\%$. The final Wasserstein-2 error for the samples is \num{5.01e-2} for our score method and \num{4.81e-2} for the BSDE method. In this example, our score method has a similar performance compared with the BSDE method. 
The third figure displays the variance change for $x_t$, demonstrating the capability of our score-based dynamics to capture density changes.
The green curve represents the true variance under the optimal trajectory, showing a decreasing trend. In contrast, the blue curve represents the variance of $x_t$ when utilizing an untrained neural network, failing to capture the decreasing variance. After training, the variance for $x_t$ aligns with the true variance, as depicted by the orange curve. Note that we have the exact initial condition, so the error of variance at $t=0$ could be significantly reduced if we increase the number of samples. Here, we have an error of about $0.01$ at $t=0$ due to a limited number of samples. 

For the $2$ dimensional case presented in the second row, the four figures adopt a format similar to those in Figure \ref{fig:static}.
The final errors for $\phi$ and its derivatives are $1.55\%$, $2.64\%$, $2.65\%$. The errors for the BSDE method are $1.68\%$, $2.64\%$, $2.62\%$. The final Wasserstein-2 error for the samples is \num{6.25e-2} for our score method and \num{3.55e-2} for the BSDE method. In order to visualize the dynamic of the particles, we present the temporal evolution of particles in figure \ref{fig:LQ2d_scatter}. The first line depicts the dynamic state \eqref{eq:ODE_x_v} with the optimal velocity field and true density function. The second line shows the state dynamic with the velocity field and the density approximated by the neural network and KDE estimate respectively. The third line displays the stochastic dynamic \eqref{eq:SDE_x} with the optimal velocity field. Our approach effectively captures the diminished variance in density, distinguishing itself from the stochastic dynamic, which exhibits a comparatively less structured behavior. 

\begin{figure*}[t!]
    \centering
    \includegraphics[width=0.242\textwidth]{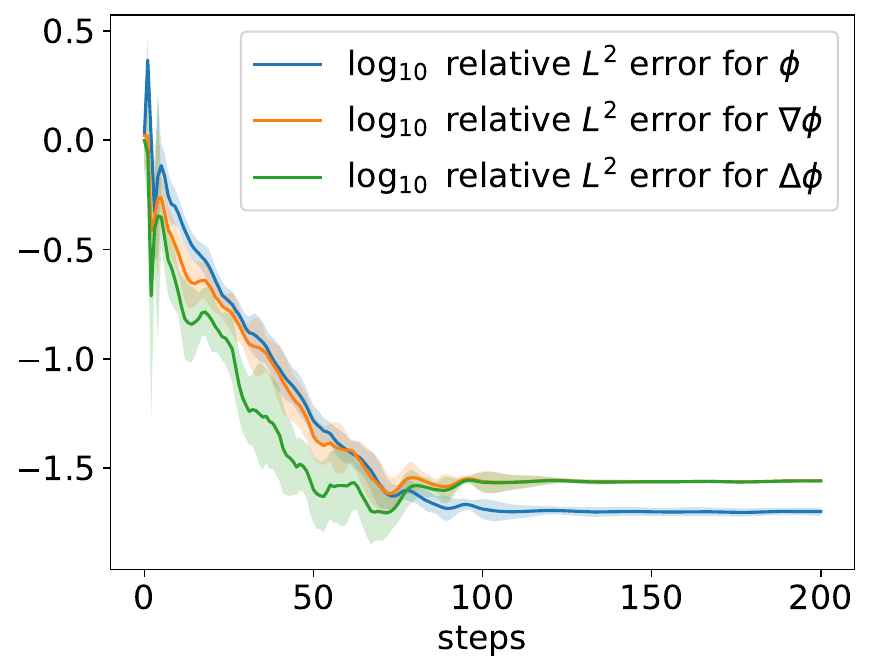}
    \includegraphics[width=0.242\textwidth]{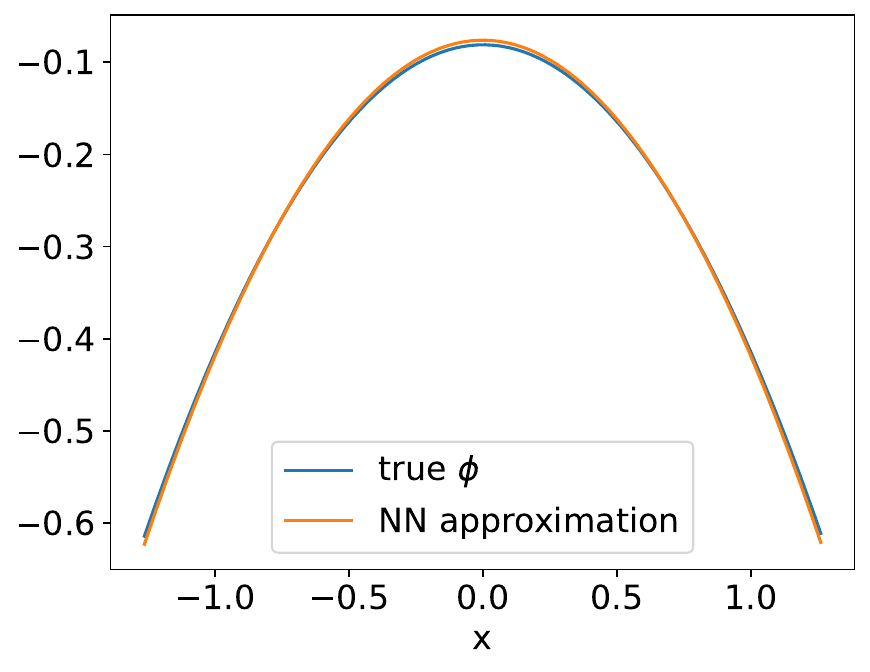}
    \includegraphics[width=0.242\textwidth]{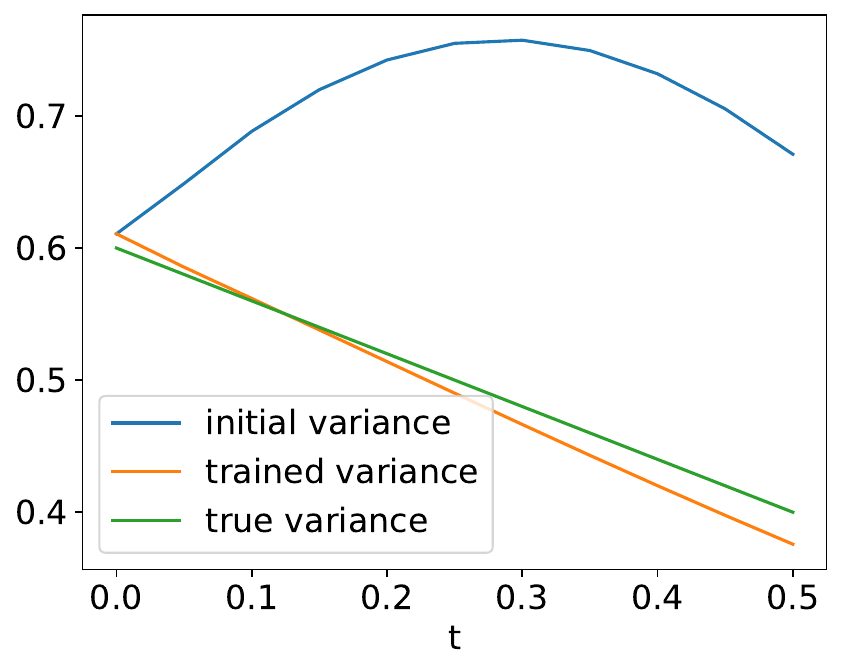}
    \includegraphics[width=0.242\textwidth]{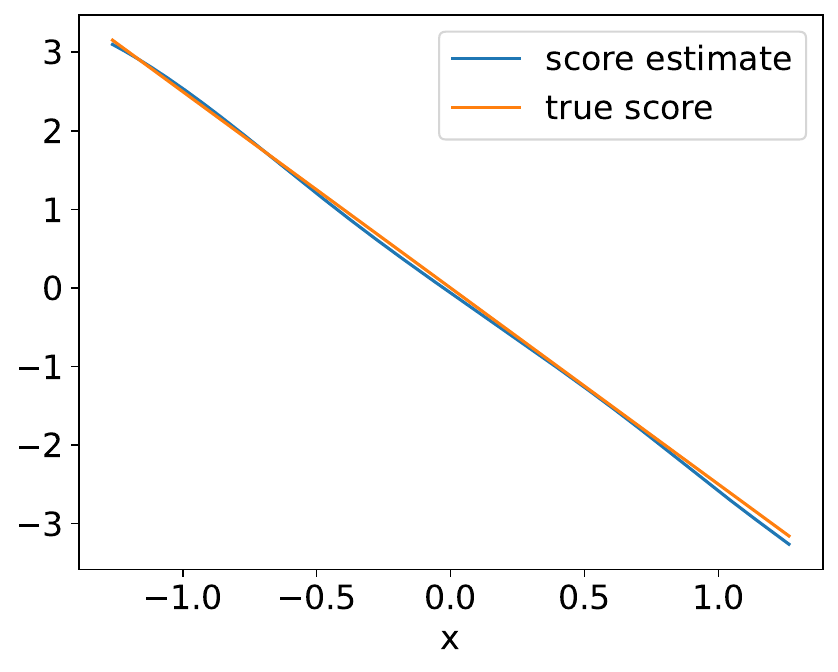}
    \includegraphics[width=0.242\textwidth]{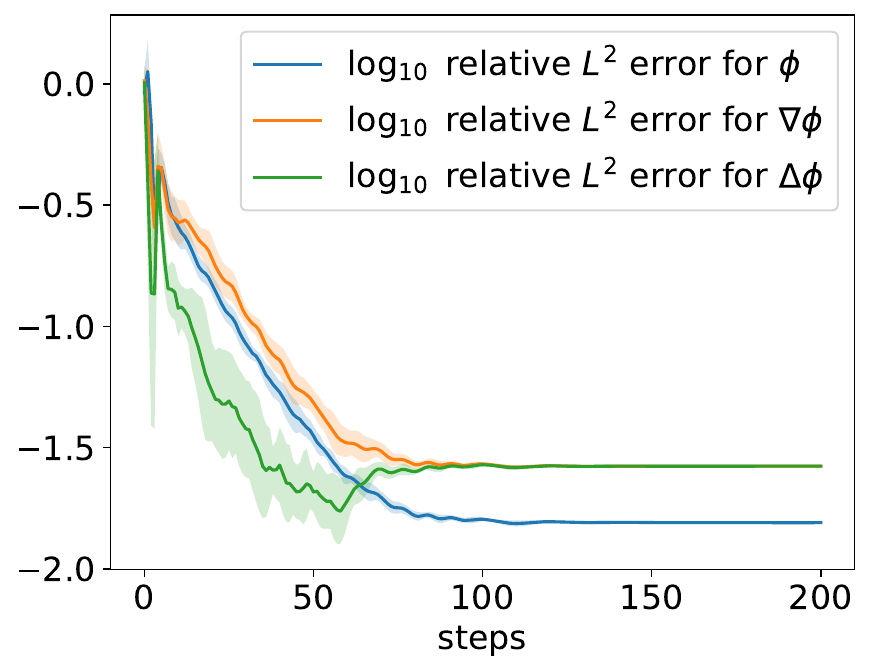}
    \includegraphics[width=0.242\textwidth]{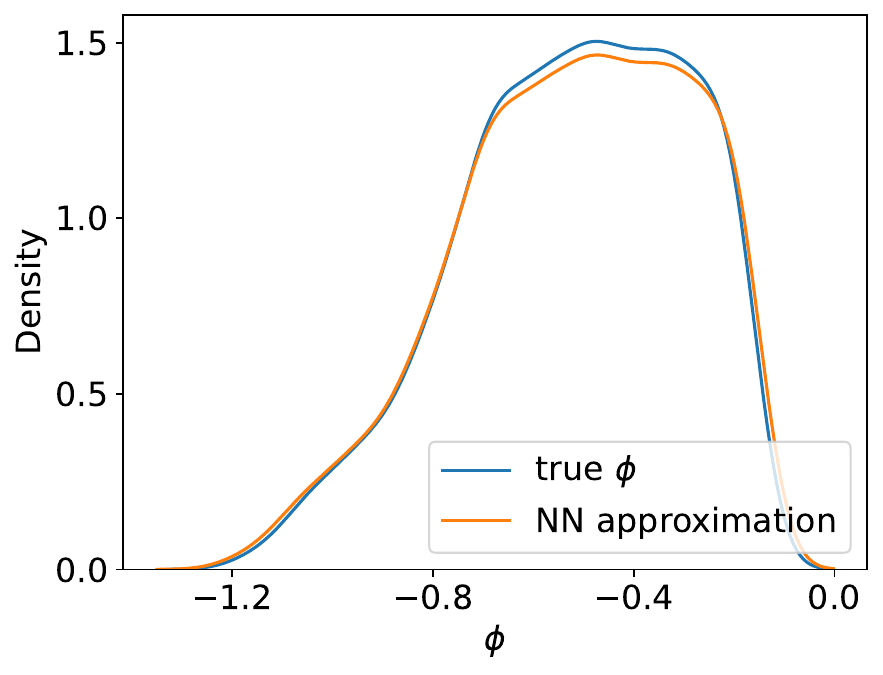}
    \includegraphics[width=0.49\textwidth]{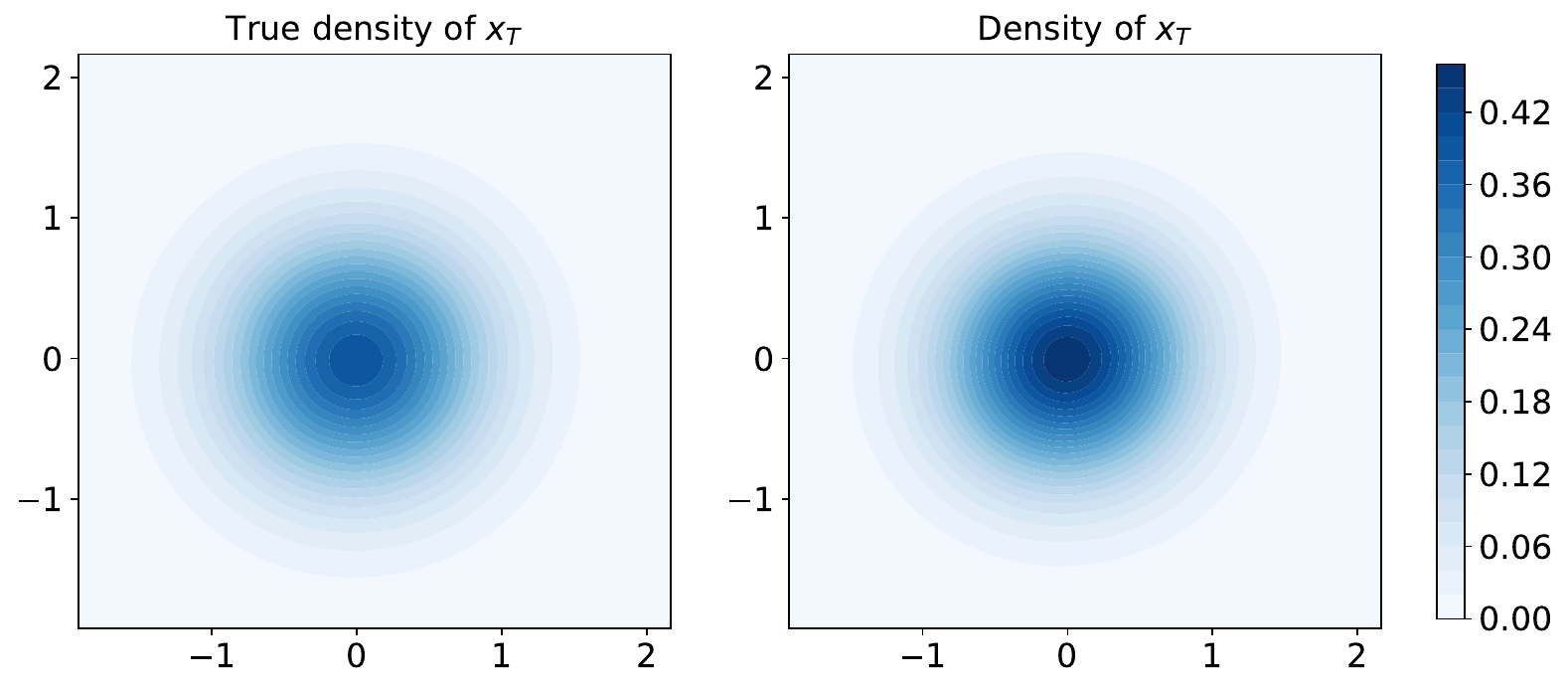}
    \caption{Numerical results for the LQ MFC problem. The first row presents results in $1$ dimension, including the training curve, plot of $\phi(0,\cdot)$, variance plot, and score plot. The second row is results for $2$ dimensions, including the training curve, density plot for $\phi(0,x_0)$, and a contour plot comparison for the density $\rho(T,x_T)$.}
    \label{fig:LQ}
\end{figure*}

\begin{figure*}[t!]
\centering
\includegraphics[width=\textwidth]{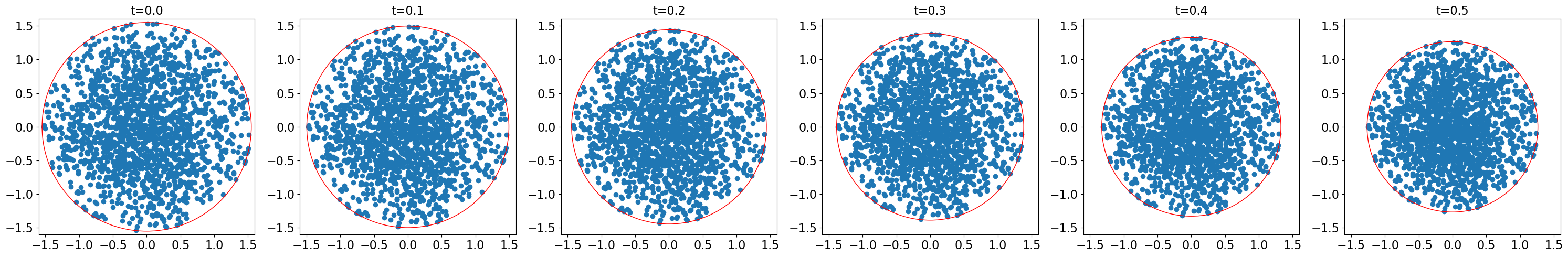}
\includegraphics[width=\textwidth]{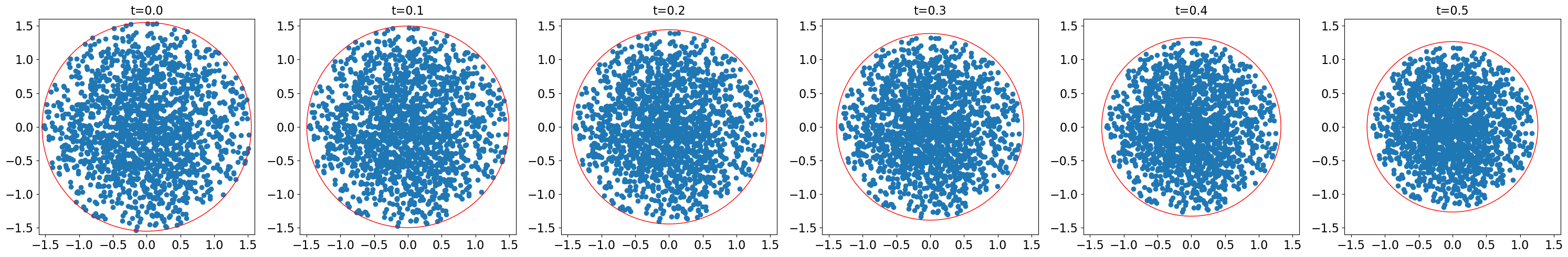}
\includegraphics[width=\textwidth]{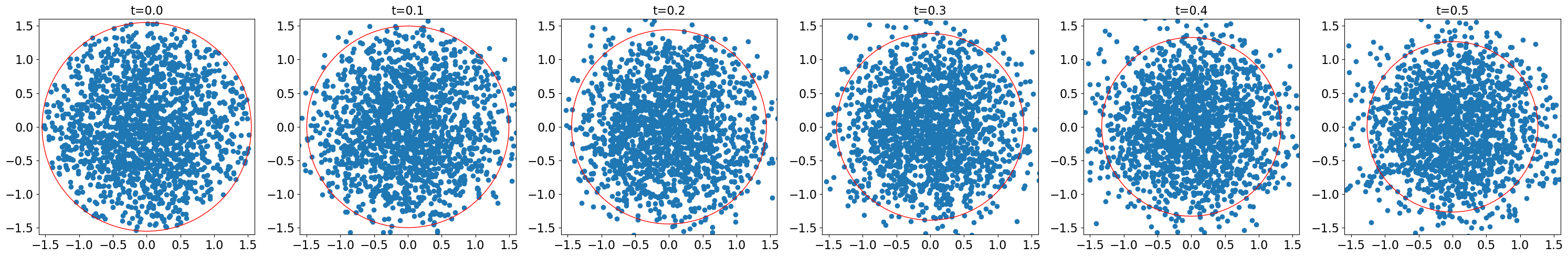}
\caption{Particle trajectories for the true score dynamic (first row), approximated score dynamic (second row), and stochastic dynamic (FBSDE) with true velocity (third row). The score dynamic demonstrates a more structured behavior.}
\label{fig:LQ2d_scatter}
\end{figure*}

\subsection{Systemic risk}\label{sec:SystemicRisk}
This example, as explored in \cite{carmona2015mean}, operates within one-dimensional space, where $x_t$ represents the logarithm of a bank's cash reserves. The model captures the interbank borrowing and lending activities through the drift. The objective is to minimize
\begin{equation}
J[\alpha] = \EE\sqbra{\int_0^T \parentheses{\frac12 \alpha_t^2 - q \alpha_t (\overline{x}_t - x_t) + \dfrac{\ve}{2} (\overline{x}_t-x_t)^2} \rd t + \dfrac{c}{2}(\overline{x}_T - x_T)^2}
\end{equation}
subject to the dynamic
\begin{equation*}
\rd x_t = \parentheses{a(\overline{x}_t-x_t) + \alpha_t} \rd t + \sigma \rd W_t,
\end{equation*}
where $\overline{x}_t = \EE x_t$  denotes the mean cash reserves. 
% The Fokker Planck equation is 
% \begin{equation*}
% \pt \rho + \nx \cdot \parentheses{\rho(a(\overline{x}_t-x) + \alpha_t)} = \frac12 \sigma^2 \Delta \rho,
% \end{equation*}
% where $\alpha_t = \nx \phi(t,x) + q(\overline{x}_t-x)$ and $\overline{x}_t = \int_\RR x \rho(t,x) \rd x$. 

The FP--HJB system governing this scenario is represented by
\begin{equation*}
\left\{\begin{aligned}
&\pt \phi + \frac12 \abs{\nx \phi}^2 + (a+q)(\overline{x}_t-x) \cdot\nx \phi - \dfrac{\ve - q^2}{2}(\overline{x}_t-x)^2 + \frac12 \sigma^2 \Delta \phi=0,\\
&\pt \rho + \nx \cdot \sqbra{\rho\parentheses{(a+q)(\overline{x}_t-x) + \nx \phi}} = \frac12 \sigma^2 \Delta \rho,\\
& \rho(0,x)=\rho_0(x),\quad \phi(T,x)=-\dfrac{c}{2} (\overline{x}_T-x)^2. 
\end{aligned}\right.
\end{equation*}
where $\overline{x}_t = \int_\RR x \rho(t,x) \rd x$. The solution of the HJB equation is expressed by
$\phi(t,x) = -\eta_t (\overline{x}_t-x)^2 -\chi_t$, where $\eta_t$ is the solution to the following Riccati equation
$$\pt \eta_t = 2(a+q) \eta_t + \eta_t^2 + q^2 - \ve, \hspace{0.2in} \eta_T=c,$$
and $\chi_t = - \frac12 \sigma^2 \int_t^T \eta_s \rd s$.

This example is a generalization of the problem described in section \ref{sec:background}, because the density $\rho$ is involved in the Lagrangian (running cost), the terminal cost, and the velocity field.
Unlike the preceding examples, this scenario lacks an explicit expression for the terminal condition $\phi(T,\cdot)$. Consequently, a direct parametrization as in \eqref{eq:NN_terminal} is not possible. To address this, an additional loss term
$$T\cdot \EE\sqbra{\parentheses{\phi(T,x_T)+\dfrac{c}{2}(\overline{x}_T-x_T)}^2} $$
is introduced during training to enforce the terminal condition. In the numerical test, $\overline{x}_t$ is approximated using the empirical mean of the sampled points.

The numerical results are presented in Figure \ref{fig:SystemicRisk1d}, adopting a layout similar to the first row of Figure \ref{fig:static}. The final errors for $\phi$ and its derivatives are $2.18\%$, $1.76\%$, $0.95\%$. The errors for BSDE method are $2.82\%$, $2.48\%$, $1.06\%$. The final Wasserstein-2 error for the samples is \num{4.10e-2} for our score method and \num{4.88e-2} for the BSDE method. So our score method performs better than the BSDE method in this example.

\begin{figure*}[t!]
\centering
\includegraphics[width=0.242\textwidth]{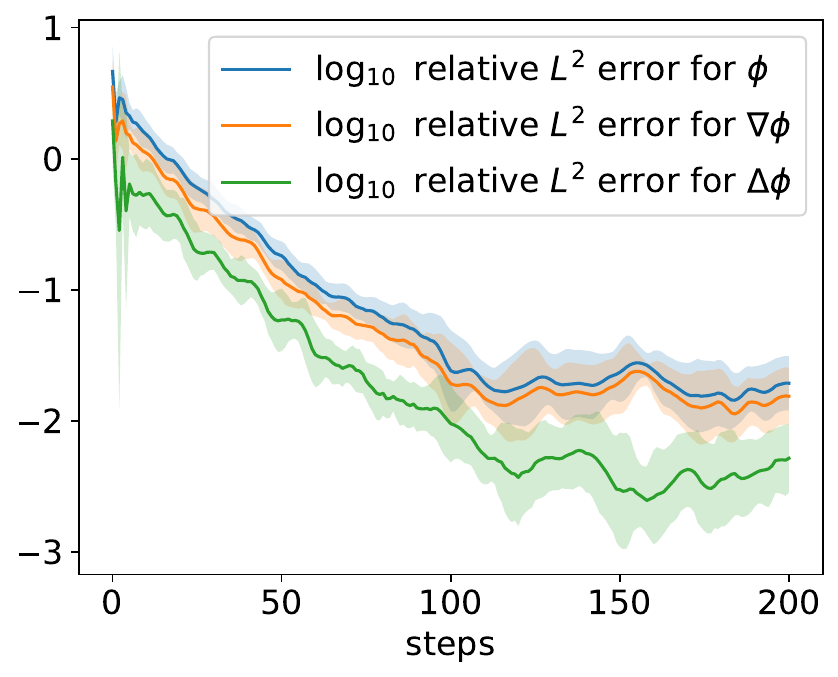}
\includegraphics[width=0.242\textwidth]{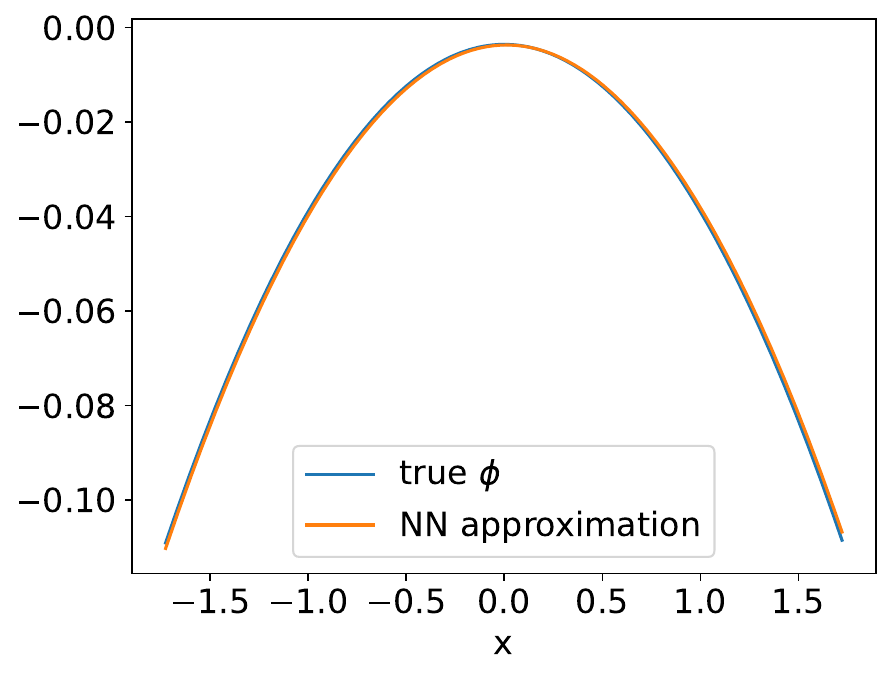}
\includegraphics[width=0.242\textwidth]{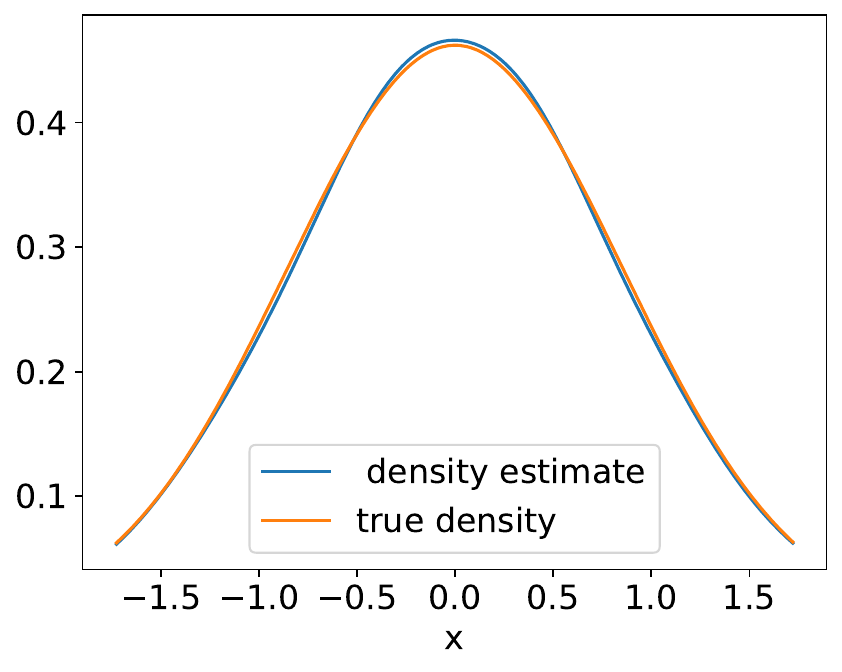}
\includegraphics[width=0.242\textwidth]{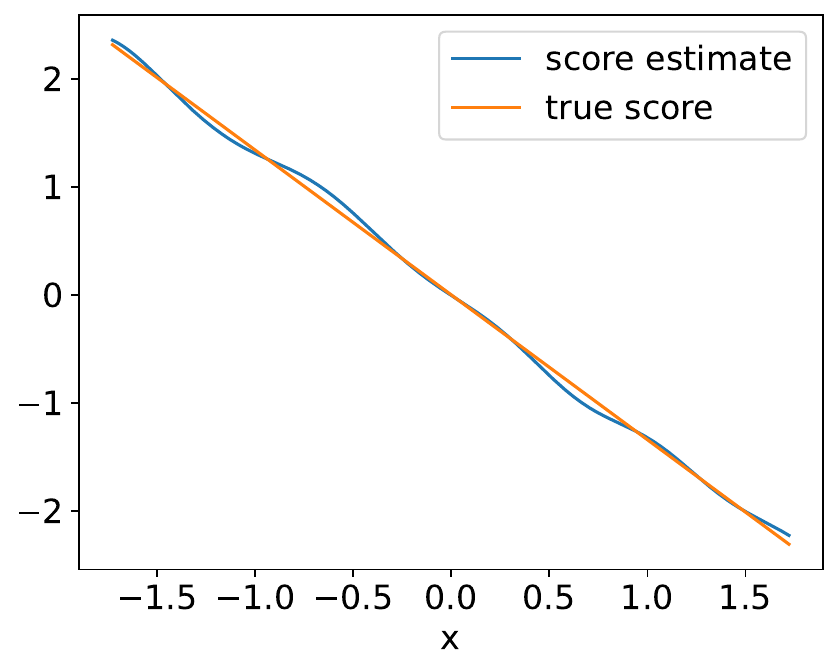}
\caption{Numerical results for the systemic risk example in $1$ dimension, including the training curve, plot of $\phi(0,\cdot)$, density plot, and score plot.}
\label{fig:SystemicRisk1d}
\end{figure*}

\section{Conclusion and future directions}
In this paper, we introduce a deep learning method to approximate the MFC problems. We first introduce forward-backward score dynamics, which formulates deterministic characteristics for the system of Fokker-Planck equations and Hamilton-Jacobi equations. In the algorithm, we apply the density estimation method to approximate the score function and use the neural network function to approximate the adjoint variable. We design a least square type loss function to fit the proposed deterministic characteristics. Numerical examples, including mean field control with entropy potential energies, linear quadratic problems, and system risks, demonstrate the effectiveness of the proposed methods.  

There are several promising directions for future research. From a theoretical perspective, an intriguing direction is to derive higher-order equations for our adjoint ODE system. Such equations could serve as a deterministic analog of the higher-order adjoint equation in FBSDE. From the numerical aspect, the accuracy of the density estimate with the Gaussian kernel is below our expectation. Further numerical studies on density estimation and a more comprehensive comparison against classical forward-backward SDE methods are important for a thorough understanding of our proposed methodology.
Finally, it is also important to generalize the algorithm to MFC problems where the diffusion matrix is under control. In this scenario, the HJB equation is fully nonlinear, resulting in a more complicated system involving score functions.

% \bmhead{Acknowledgements}
\textbf{Acknowledgements}
M. Zhou is partially supported by AFOSR YIP award No. FA9550-23-1-0087. M. Zhou and S. Osher are partially funded by AFOSR MURI FA9550-18-502 and ONR N00014-20-1-2787. Wuchen Li’s work is partially supported by AFOSR YIP award No. FA9550-23-1-0087, NSF DMS-2245097, and NSF RTG: 2038080.

\textbf{Data availability statement}
There is no data involved in this paper. Our algorithm is based on the packages in Python.  
\bibliography{ref}

\begin{appendices}
\section{Details for the numerical implementation}
% Put all parameters and network structure here.
For all the examples in Section \ref{sec:example}, we employ a fully connected network with $2$ layers and $30$ nodes in each layer, with the square of ReLU as the activation function. For the first two examples, we implement a hard parametrization of the terminal condition (cf. \eqref{eq:NN_terminal}). For all the examples, we set $N_t=10$, $k_{end}=200$. The number of validation samples is $1000d$ for all the examples. The default parameters for the PyTorch Adam optimizer are utilized for model training. Detailed parameters for all examples are provided in table \ref{tab:parameters}. We use the same network structure and same optimization scheme for our score method and the BSDE method.

\begin{table}[ht]
\centering
\resizebox{\textwidth}{!}{
\begin{tabular}{c|ccccc}
\hline
 & $T$ & Learning rate & $\sigma_K$ & batch size & others \\
 \hline
 potential energy 1d & 0.5 & 0.02 & 0.35 & 200 & $\gamma=0.1$ \\
 potential energy 2d & 0.5 & 0.1 & 0.4 & 1000 & $\gamma=0.1$ \\
 LQ 1d & 0.5 & 0.1 & 0.35 & 200 & $\beta=5$, $\gamma=0.1$ \\
 LQ 2d & 0.5 & 0.1 & 0.4 & 1000 & $\beta=5$, $\gamma=0.1$ \\
 Systemic risk & 0.1 & 0.02 & 0.3 & 400 & $\sigma=1$, $q=0.5$, $q=\ve=0.1$ \\
 \hline
\end{tabular}
}
\caption{Parameters for the examples.}
\label{tab:parameters}
\end{table}

We elucidate the computation of the Wasserstein-2 distance between the samples $\{x^{(i)}_T\}_{i=1}^N$ and the density $\rho(T,\cdot)$. Directly computing the Wasserstein distance between them is non-trivial and computationally expensive. Alternatively, we compute the mean and covariance of the samples $\{x^{(i)}_T\}_{i=1}^N$, subsequently determining the Wasserstein-2 distance between two Gaussian distributions: one with the computed mean and covariance, and the other representing the true density. In comparison, we also calculate the systemic error, defined as the Wasserstein-2 distance between the true density and the estimated Gaussian. The mean and covariance for this estimation are derived from data directly sampled from the true density. The Wasserstein-2 error and the corresponding systemic error are listed in table \ref{tab:Wasserstein}.

\begin{table}[ht]
\centering
\begin{tabular}{c|ccc}
\hline
 &  score method & BSDE method & systemic error\\
 \hline
 potential energy 1d& \num{6.41e-2} & \num{8.23e-2} & \num{4.20e-2}\\
 potential energy 2d& \num{9.18e-2} &\num{5.87e-2} & \num{4.28e-2}\\
 LQ 1d& \num{5.01e-2} &\num{4.81e-2} & \num{2.59e-2}\\
 LQ 2d& \num{6.25e-2} &\num{3.55e-2} & \num{2.64e-2}\\
 Systemic risk& \num{4.10e-2} &\num{4.88e-2} & \num{3.08e-2}\\
 \hline
\end{tabular}
\caption{Wasserstein-2 error and the systemic errors for all the examples. Each error is the average of multiple runs.}
\label{tab:Wasserstein}
\end{table}

Here, we remark that we also tested the Wasserstein-2 distance errors between $\{x^{(i)}_T\}_{i=1}^N$ and data sampled from the true distribution, with the help of the package \cite{flamary2021pot}. However, we observed that the systemic error, defined as the Wasserstein-2 distance between two sets of independent samples from the true distribution, was significantly large due to the limited sample size.

\end{appendices}

\end{document}